\documentclass[11 pts]{amsart}

\usepackage{amssymb, amsmath, amsfonts, amsthm, graphics, hyperref,caption, subcaption, enumerate, mathtools, url}
\usepackage{tikz}
\usetikzlibrary{arrows}
\usetikzlibrary{calc}
\usepackage[all]{xy}

\newcommand{\integers}{\mathbb{Z}}
\newcommand{\complexes}{\mathbb{C}}
\DeclareMathOperator{\Gr}{Gr}
\DeclareMathOperator{\Span}{Span}
\newcommand{\Lm}{\Lambda}
\newcommand{\cM}{\mathcal{M}}
\newtheorem{thm}{Theorem}[section]
\newtheorem{prop}[thm]{Proposition}
\newtheorem{cor}[thm]{Corollary}
\newtheorem{lem}[thm]{Lemma}
\newtheorem{defn}[thm]{Definition}
\DeclareMathOperator{\Bd}{Bound}
\newcommand{\lm}{\lambda}
\newcommand{\bdot}[2]{
\draw [black, fill = black] (#1, #2) circle [radius = 0.1];
}

\newcommand{\wdot}[2]{
\draw [black, fill = white] (#1, #2) circle [radius = 0.1];
}

\newcommand{\edge}[3]{
\draw (#1,#2) -- (#1, #2 + #3);
\bdot{#1}{#2}
\wdot{#1}{#2 + #3}
}

\usepackage{tikz}
\usetikzlibrary{arrows}
\usetikzlibrary{calc}

\title{Combinatorics of Symmetric Plabic Graphs}
\author{Rachel Karpman \and Yi Su}

\begin{document}

\begin{abstract}
A \emph{plabic graph} is a planar bicolored graph embedded in a disk, which satisfies some combinatorial conditions.  Postnikov's \emph{boundary measurement map} takes the space of positive edge weights of a plabic graph $G$ to a \emph{positroid cell} in a totally nonnegative Grassmannian.  In this note, we investigate plabic graphs which are symmetric about a line of reflection, up to reversing the colors of vertices.  These \emph{symmetric plabic graphs} arise naturally in the study of total positivity for the Lagrangian Grassmannian.  We characterize various combinatorial objects associated with symmetric plabic graphs, and describe the subset of a Grassmannian which can be realized by \emph{symmetric weightings} of symmetric plabic graphs.  
\end{abstract}


\maketitle

\section{Introduction}

A plabic graph is a planar graph embedded in a disk, with vertices colored black or white.  Postnikov introduced plabic graphs as a tool for studying the \emph{positroid stratification} of the \emph{totally nonnegative Grassmannian} $\Gr_{\geq 0}(k,n)$ \cite{Pos06}.

The totally nonnegative Grassmannian is the region of the Grassmannian $\Gr(k,n)$ where all Pl\"ucker coordinates are nonnegative real numbers, up to multiplication by a common scalar.  Postnikov defined a stratification of $\Gr_{\geq 0}(k,n)$ by \emph{positroid cells}, each defined as the locus in $\Gr_{\geq 0}(k,n)$ where some set of Pl\"ucker coordinates vanish.  The resulting \emph{positroid stratification} of $\Gr_{\geq 0}(k,n)$ has a rich geometric and combinatorial structure.  There are numerous combinatorial objects which index positroid cells, including \emph{bounded affine permutations}, \emph{Grassmann necklaces}, and a class of matroids called \emph{positroids} \cite{Pos06,KLS13}.  

Given a plabic graph $G$, Postnikov's \emph{boundary measurement map} takes the space of positive real edge weights of $G$ surjectively to a positroid cell $\Pi_G$ in $\Gr_{\geq 0}(k,n)$ for some values of $k$ and $n$.  Moreover, the bounded affine permutation, Grassmann necklace, and matroid of $\Pi_G$ are encoded in structure of $G$.  Note that plabic graphs are not in bijection with positroid cells; rather, for each positroid cell, we have a family of plabic graphs. 

In this note, we study plabic graphs which satisfy a symmetry condition.  See Figure \ref{ex} for an example.  These \emph{symmetric plabic graphs} arise in the theory of total positivity for the Lagrangian Grassmannian $\Lm(2n)$, the moduli space of maximal isotropic subspaces with respect to a symplectic bilinear form.  The connection with $\Lm(2n)$, as well as additional results on the combinatorics of symmetric plabic graphs, will be discussed in a forthcoming paper by the first author \cite{Kar15}.  

\begin{figure}
\centering
\begin{tikzpicture}[scale = 0.75]
\draw [gray!30] (0,3.5) -- (0,-3.5);
\draw (0,0) circle (3);
\draw ({2*cos(60)},{2*sin(60)}) -- ({-2*cos(60)},{2*sin(60)})  -- ({-2*cos(60)},{-2*sin(60)})  -- ({2*cos(60)},{-2*sin(60)}) --  ({2*cos(60)},{2*sin(60)});
\draw ({2*cos(60)},{2*sin(60)}) -- ({3*cos(60)},{3*sin(60)});
\draw ({-2*cos(60)},{2*sin(60)}) -- ({-3*cos(60)},{3*sin(60)});
\draw ({-2*cos(60)},{-2*sin(60)}) -- ({-3*cos(60)},{-3*sin(60)});
\draw ({2*cos(60)},{-2*sin(60)}) -- ({3*cos(60)},{-3*sin(60)});
\draw ({3*cos(30)},{3*sin(30)}) -- (2,0) -- ({3*cos(30)},{-3*sin(30)});
\draw ({-3*cos(30)},{3*sin(30)}) -- (-2,0) -- ({-3*cos(30)},{-3*sin(30)});
\bdot{{2*cos(60)}}{{2*sin(60)}};\wdot{{-2*cos(60)}}{{2*sin(60)}};\bdot{{-2*cos(60)}}{{-2*sin(60)}};\wdot{{2*cos(60)}}{{-2*sin(60)}};
\wdot{{3*cos(60)}}{{3*sin(60)}};\bdot{{3*cos(30)}}{{3*sin(30)}};\wdot{2}{0};\bdot{{3*cos(30)}}{{-3*sin(30)}};\bdot{{3*cos(60)}}{{-3*sin(60)}};
\bdot{{-3*cos(60)}}{{3*sin(60)}};\wdot{{-3*cos(30)}}{{3*sin(30)}};\bdot{-2}{0};\wdot{{-3*cos(30)}}{{-3*sin(30)}};\wdot{{-3*cos(60)}}{{-3*sin(60)}};
\node [above right] at ({3*cos(60)},{3*sin(60)}) {1};
\node [above right] at ({3*cos(30)},{3*sin(30)}) {2};
\node [below right] at ({3*cos(-30)},{3*sin(-30)}) {3};
\node [below right] at ({3*cos(-60)},{3*sin(-60)}) {4};
\node [below left] at ({3*cos(240)},{3*sin(240)}) {5};
\node [below left] at ({3*cos(210)},{3*sin(210)}) { 6};
\node [above left] at ({3*cos(150)},{3*sin(150)}) {7};
\node [above left] at ({3*cos(120)},{3*sin(120)}) {8};
\end{tikzpicture}
\caption{A symmetric plabic graph.}
\label{ex}
\end{figure}
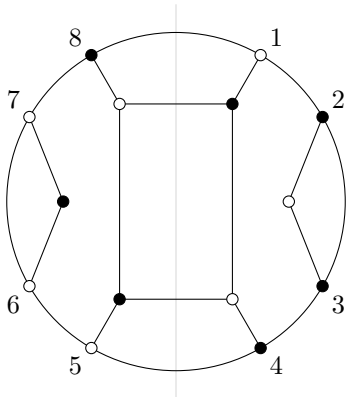

Let $G$ denote a plabic graph, and let $N$ denote the underlying uncolored network. Suppose $N$ is symmetric with respect to reflection through a distinguished diameter $d$ of the disc. Let  $V$ be the vertex set of $E$, and let $r:V \rightarrow V$ map each vertex $v \in V$ to its mirror image across the line $d$.  Then $G$ is a \emph{symmetric plabic graph} if the following conditions hold:
\begin{enumerate}
\item $G$ has no vertices on the line $d$, although edges may cross $d$
\item For each $v \in V$, the vertices $v$ and $r(v)$ have opposite colors.
\end{enumerate}

We briefly summarize our results. In Theorem \ref{graphsLagrangian}, we characterize the positroids, Grassmann necklaces, and bounded affine permutations associated with symmetric plabic graphs.  We then consider \emph{symmetric weightings} of symmetric plabic graphs; that is, weightings where each edge $(u,v)$ has the same weight as its reflection $(r(u),r(v))$.  Theorem \ref{allpt} and Corollary \ref{both} give a complete description of the set of points in $\Gr_{\geq 0}(k,n)$ corresponding to such weightings of symmetric plabic graphs.  This is a subvariety, cut out set-theoretically by linear equations, which we call the \emph{symmetric part} of $Gr(k,n)$.  With the right choice of conventions, the symmetric part of $\Gr_{\geq 0}(k,n)$ is precisely the totally nonnegative part of the Lagrangian Grassmannian \cite{Kar15}.   Finally, Theorem \ref{symbrd} gives an explicit construction which yields a symmetric weighting of a symmetric plabic graph for each point in the symmetric part of $\Gr_{\geq 0}(k,n)$.  In particular, we construct a symmetric \emph{bridge graph} for each point in $\Gr_{\geq 0}(k,n)$.  Bridge graphs are a special class of plabic graphs which appear as a computational tool in particle physics \cite{ABCGPT14}.

\section{Background}

\subsection{Notation}

For natural numbers $k \leq n$, let $[n]$ denote the set $\{1,2,\ldots,n\}$, and let ${{[n]}\choose{k}}$ denote the set of all $k$-element subsets of $[n]$.  For $a \in [n]$, let $\leq_a$ denote the cyclic shift of the usual linear order on $n$ given by 
\begin{equation}a < a + 1 < \ldots < n < 1 < \ldots < a-1.\end{equation}
Note that $\leq_1$ is the usual order $\leq$.  We extend this to a partial order on ${{[n]}\choose{k}}$, by setting $I \leq_a J$ if we have $i_{\ell} \leq_a j_{\ell}$ for all $\ell \in [k]$, where 
\begin{equation}I = \{i_1 <_a i_2 <_a \ldots <_a i_k\} \text{ and }J = \{j_1 <_a j_2 <_a \ldots <_a j_k\}.\end{equation}
For $a,b \in [n]$, we define the \emph{cyclic interval} $[a,b]^{cyc}$ by
\begin{equation}[a,b]^{cyc} = \begin{cases}
\{a,a+1,\ldots,b\} & a \leq b\\
\{a,a+1,\ldots,n-1,n,1 \ldots,b\} & a > b
\end{cases}.\end{equation}
(Note that this differs slightly from Postnikov's convention.)  If we arrange the elements of $[n]$ clockwise around a circle, then $[a,b]^{cyc}$ represents a sequence of consecutive numbers. 

Let $S_n$ be the symmetric group in $n$ letters, and let $s_i$ denote the simple transposition $(i,i+1)$ which switches $i$ and $i+1$.  Let $\mathbb{R}_{>0}$ denote the positive reals.

\subsection{Grassmannians and Pl\"{u}cker coordinates}

Let $\Gr(k,n)$ denote the Grassmannian of $k$-dimensional linear subspaces of the vector space $\complexes^n$.  We may realize $\Gr(k,n)$ as the space of full-rank $k \times n$ matrices modulo the left action of $\text{GL}(k)$, the group of invertible $k\times k$ matrices; a matrix $M$ represents the space spanned by its rows.  

The \emph{Pl{\"u}cker embedding}, which we denote $p$, maps $\Gr(k,n)$ into the projective space 
\begin{math}\mathbb{P}^{ {n\choose k}-1}\end{math}
with homogeneous coordinates $x_J$ indexed by the elements of 
\begin{math}{{[n]}\choose k}.\end{math}
For 
\begin{math} J \in {[n] \choose k}\end{math}
let $\Delta_J$ denote the minor with columns indexed by $J$.  Let $V$  be a $k$-dimensional subspace of $\complexes^n$ with representative matrix $M$.  Then $p(V)$ is the point defined by 
\begin{math} x_J = \Delta_J(M)\end{math}.
This map embeds $\Gr(k,n)$ as a smooth projective variety in 
\begin{math} \mathbb{P}^{ {n\choose k}-1}.\end{math}
The homogeneous coordinates $\Delta_J$ are known as \emph{Pl{\"u}cker  coordinates} on $\Gr(k,n)$.  The \emph{totally nonnegative Grassmannian}, denoted $\Gr_{\geq 0}(k,n),$ is the subset of $\Gr(k,n)$ whose Pl{\"u}cker coordinates are all nonnegative real numbers, up to multiplication by a common scalar. 

Let \begin{math}V \in \Gr_{\geq 0}(k,n)\end{math}.
The indices of the non-vanishing Pl{\"u}cker  coordinates of $V$ give a set
\begin{math} \mathcal{J} \subseteq {{[n]}\choose k}\end{math}
called the \emph{matroid} of $V$.  We define the \emph{matroid cell} 
\begin{math} \mathcal{M}_{\mathcal{J}}\end{math}
as the locus of points $V \in \Gr_{\geq 0}(k,n)$ with matroid $\mathcal{J}$.  The matroids $\mathcal{J}$ for which $\mathcal{M}_{\mathcal{J}}$ is nonempty are called \emph{positroids}, and the corresponding matroid cells are called positroid cells. Positroid cells form a \emph{stratification} of $\Gr_{\geq 0}(k,n)$.   That is, the closure of a positroid cell $\Pi$ in $\Gr_{\geq 0}(k,n)$ is the union of $\Pi$ and some lower-dimensional positroid cells \cite{Pos06}.

\subsection{Plabic graphs}

A \emph{plabic graph} is a planar graph embedded in a disk, with each vertex colored black or white.  The boundary vertices are numbered $1,2,\ldots,n$ in clockwise order, and all boundary vertices have degree one. We call the edges adjacent to boundary vertices \emph{legs} of the graph, and a leaf adjacent to a boundary vertex a \emph{lollipop}.  A \emph{white lollipop} is a white leaf adjacent to a black boundary vertex, while a \emph{black} lollipop is the opposite. 

Postnikov introduced plabic graphs in \cite[Section 11.5]{Pos06}.  We follow the conventions of \cite{Lam13b}, which are more restrictive than Postnikov's. In particular, we require a plabic graph to be bipartite, with the black and white vertices forming the partite sets.  An \emph{almost perfect matching} on a plabic graph is a subset of its edges which covers each interior vertex exactly once; boundary vertices may or may not be covered.  We consider only plabic graphs which admit an almost perfect matching.  Finally, we require that no edge in a plabic graph connects two boundary vertices.

We define a collection of paths and cycles in $G$, called \emph{trips}, as follows.  We begin by traversing an edge $\{u,v\}$ of $G$, from $u$ to $v$. We then proceed according to the \emph{rules of the road}: turn (maximally) left at every white internal vertex, and (maximally) right at every black internal vertex. Continuing in this fashion, we eventually reach a boundary vertex.  The resulting directed path is a \emph{trip} in $G$.  See Figure \ref{trip} for an example.  We repeat this process for every boundary vertex.  If the resulting collection of trips covers every edge of $G$ twice, once in each direction, we are done.  

Otherwise, we find an internal edge $e=\{u,v\}$ such that no trip covers $e$ in the direction $u \rightarrow v$.  We begin by tracing $e$ in this direction, and proceed according to the rules of the road, until we find ourselves one against about to trace the edge $u \rightarrow v$.  The resulting directed cycle is a trip. Repeat this process until each edge of $G$ is covered twice by trips, once in each direction.  

\begin{figure}
\centering
\begin{tikzpicture}[scale = 0.75]
\draw (0,0) circle ({4*cos{45}});
\draw (-1,1) -- (1,1) -- (1,-1) -- (-1,-1) -- (-1,1);
\draw (1,1) -- (2,2);
\draw (-1,1) -- (-2,2);
\draw (-1,-1) -- (-2,-2);
\draw (1,-1) -- (2,-2);
\bdot{1}{1};\wdot{-1}{1};\bdot{-1}{-1};\wdot{1}{-1};
\wdot{2}{2};\bdot{-2}{2};\wdot{-2}{-2};\bdot{2}{-2};
\draw [dashed, ->] (0.8,-1.2) -- (-0.8,-1.2);
\draw [dashed, ->] (-1.2,-0.8) -- (-1.2,0.8);
\draw [dashed, ->] (1.72,-2) -- (1,-1.28);
\draw [dashed, ->] (-1.28,1) -- (-2,1.72);
\node [above right] at (2,2) {$1$};
\node [below right] at (2,-2) {$2$};
\node [below left] at (-2,-2) {$3$};
\node [above left] at (-2,2) {$4$};
\end{tikzpicture}
\caption{A trip in a plabic graph $G$.  We have $\bar{f}_G(2) = 4$.}
\label{trip}
\end{figure}
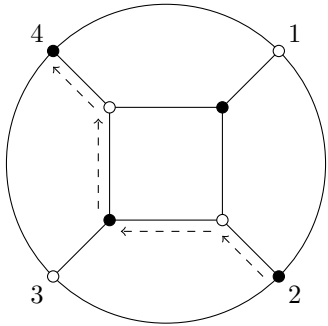

Given a plabic graph $G$ with $n$ boundary vertices, we define the \emph{trip permutation} 
\begin{math} \bar{f}_G \in S_n\end{math}
of $G$ by setting 
\begin{math} \bar{f}_G(a) = b\end{math}
if the trip that starts at boundary vertex $a$ ends at boundary vertex $b$.  

The boundary measurement map is defined only for \emph{reduced} plabic graphs.  Postnikov defined reduced plabic graphs in terms of certain local transformations of plabic graphs, and gave a criterion for a plabic graph $G$ to be reduced \cite[Section 13]{Pos06}. We take this criterion as the \emph{definition} of a reduced graph.

\begin{defn}
A plabic graph $G$ is \emph{reduced} if it satisfies the following criteria:
\begin{enumerate}
\item $G$ has no trips which are cycles.
\item $G$ has no leaves, except perhaps some which are adjacent to boundary vertices. 
\item No trip uses the same edge twice, once in each direction, unless that trip starts (and ends) at a boundary vertex connected to a leaf.
\item No trips $T_1$ and $T_2$ share two edges $e_1,e_2$ such that $e_1$ comes before $e_2$ in both trips.  
\end{enumerate}
\end{defn}

If $G$ is a reduced graph, each fixed point of $\bar{f}_G$ corresponds to a lollipop \cite{Pos06}.  

\subsection{The boundary measurement map}

Let $G$ be a reduced plabic network with $e$ edges, and assign weights 
\begin{math}t_1,\ldots,t_e\end{math}
to the edges of $G$.  Postnikov defined a surjective map from the space of positive real edge weights of $G$ to some positroid cell $\Pi_G$
in $\Gr_{\geq 0}(k,n)$, called the \emph{boundary measurement map} \cite[Section 11.5]{Pos06}.   Postnikov, Speyer and Williams re-cast this construction in terms of almost perfect matchings \cite[Section 4-5]{PSW09}, an approach Lam developed further in \cite{Lam13b}.  We use the definition of the boundary measurement map found in \cite{Lam13b}. 

As mentioned above, an \emph{almost perfect matching} of a plabic graph $G$ is a collection of edges of $G$ which covers each internal vertex exactly once.  (Boundary vertices may or may not be covered.)
For $P$ an almost perfect matching on a plabic graph $G$ with $e$ edges, let 
\begin{equation}
\begin{split}
\partial(P)   =  & \{\text{black boundary vertices used in $P$}\} \\
&  \cup \{\text{white boundary vertices \emph{not} used in $P$}\}
\end{split}
\end{equation}
We define the \emph{boundary measurement map} 
\begin{equation}\partial_G: \left(\mathbb{R}_{> 0}\right)^e \rightarrow \mathbb{P}^{{n \choose k}-1}\end{equation}
to be the map which sends 
\begin{math}(t_1,\ldots,t_e)\end{math}
to the point with homogeneous coordinates 
\begin{equation}\Delta_J = \sum_{\partial(P)=J} t^P\end{equation}
where the sum is over all matchings $P$ of $G$, and $t^P$ is the product of the weights of all edges used in $P$.

The boundary measurement map $\partial_G$ is surjective onto the totally nonnegative cell $\Pi_G$.  However, it is almost never injective, due to the existence of \emph{gauge transformations}.  Let $v$ be a vertex of $G$, and let $\omega$ be a weighting of $G$ with $\partial_G(\omega) = X$.  Multiplying the weights of all edges incident at $v$ by $\lm \in \mathbb{R}_{> 0}$ simply multiplies all Pl\"ucker coordinates by $\lm$, and hence yields a new weighting $\omega'$ with $\partial_G(\omega') = X$.  Conversely, if two weightings of $G$ map to the same point in $\Gr_{\geq 0}(k,n)$, then they are related by some sequence of gauge transformations \cite{Pos06}.

\subsection{Bounded affine permutations}

Bounded affine permutations are one of many families of combinatorial objects which index positroid cells.  There is a natural bijection between bounded affine permutations and Postnikov's \emph{decorated permutations}, so Postnikov's results about the latter translate easily to statements about the former \cite{Pos06, KLS13}.  

\begin{defn} An \textbf{affine permutation} of order $n$ is a bijection 
\begin{math}f:\integers \rightarrow \integers\end{math}
which satisfies the condition 
\begin{equation} \label{affine} f(i+n) = f(i)+n \end{equation}
 for all 
 \begin{math} i \in \integers.\end{math}
 The affine permutations of order $n$ form a group, which we denote 
 \begin{math} \widetilde{S}_n.\end{math}
 \end{defn}

We may embed $S_n$ in $\widetilde{S}_n$ by extending each permutation periodically, in accordance with \eqref{affine}. 
Conversely, for $f$ an affine permutation, there is a unique permutation $\bar{f} \in S_n$ such that for all $i \in [n]$, we have $\bar{f}(i) \cong f(i) \pmod{n}$.

\begin{defn}
An affine permutation $f$ is a \emph{bounded affine permutation} of type $(k,n)$ if it satisfies the following conditions
\begin{enumerate}
\item  $\displaystyle \frac{1}{n} \sum_{i=1}^n (f(i)-i)=k.$
\item $i \leq \bar{f}(i) \leq i+n \text{ for all }i \in \integers.$
\end{enumerate}
We write $\Bd(k,n)$ for the set of all bounded affine permutations of type $(k,n)$.  
\end{defn}

Let $G$ be a reduced plabic graph with $n$ boundary vertices, and trip permutation $\bar{f}_G$.  Suppose we have
\begin{equation}k = |\{i \in [n] : \bar{f}_G(i) < i \text{ or $i$ is a white lollipop of $G$}\}|.\end{equation}
Then $G$ has an associated bounded affine permutation $f_G$ of type $(k,n)$ defined by setting 
\begin{equation}f_G(i) = \begin{cases}\bar{f}_G(i) & \bar{f}_G(i) > i \text{ or $G$ has a black lollipop at $i$}\\
\bar{f}_G(i)+n &  \bar{f}_G(i) < i \text{ or $G$ has a white lollipop at $i$}
\end{cases}
\end{equation}
for $i \in [n],$ and extending periodically using \eqref{affine}.  For $f$ a bounded affine permutation, we say $f$ has a \emph{white fixed point} at $i$ if $f(i) = i+n$, and that $f$ has a \emph{black fixed point} at $i$ if $f(i) = i$.

Bounded affine permutations of type $(k,n)$ are in bijection with positroid cells in $\Gr_{\geq 0}(k,n)$ \cite{Pos06}.  Let $G$ be as above, and weight the edges of $G$ with indeterminates.  The \emph{boundary measurement map} carries the space of positive real edge weights of $G$ to the positroid cell $\Pi_G$ corresponding to $\bar{f}_G$ \cite{Pos06}.  There is a family of reduced graphs for each bounded affine permutation, and hence for each positroid cell.

We represent bounded affine permutations visually using \emph{chord diagrams}, introduced in \cite[Section 16]{Pos06}.  A chord diagram for a bounded affine permutation $f$ is a circle with vertices labeled $1,2,\ldots,n$ in clockwise order.  We draw arrows from vertex $i$ to vertex $\bar{f}(i)$ for all $i \in [n]$.  By convention, if $i$ is a white fixed point of $f$, we draw a clockwise loop at $i$; if $i$ is a black fixed point, we draw a counter-clockwise loop.  

Let $(i,\bar{f}(i))$ and $(j,\bar{f}(j))$ be a pair of chords in the chord diagram of a bounded affine permutation $f$.  This pair represents a crossing if $\bar{f}(j) \in [i,\bar{f}(i)]^{cyc}$ and $j \in [\bar{f}(i),i]^{cyc}$, so the two chords intersect.  See Figure \ref{crossings}.

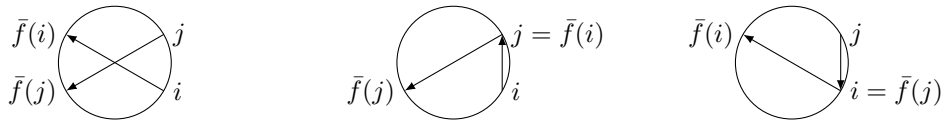
\begin{figure}
\centering
\begin{tikzpicture}[scale = 0.75]
\draw (0,0) circle (1);
\draw[->, >=latex] ({cos(30)},{sin(30)}) to (-{cos(30)},-{sin(30)});
\draw[->, >=latex] ({cos(-30)},{sin(-30)}) to (-{cos(-30)},-{sin(-30)});
\node[right] at ({cos(30)},{sin(30)}){$j$};
\node[left] at (-{cos(30)},{sin(30)}) {$\bar{f}(i)$};
\node[right] at ({cos(-30)},{sin(-30)}) {$i$};
\node[left] at (-{cos(-30)},{sin(-30)}) {$\bar{f}(j)$};
\begin{scope}[xshift = 6 cm]
\draw (0,0) circle (1);
\draw[->, >=latex] ({cos(30)},{sin(30)}) to (-{cos(30)},-{sin(30)});
\draw[->, >=latex] ({cos(-30)},{sin(-30)}) to ({cos(30)},{sin(30)});
\node[right] at ({cos(30)},{sin(30)}){$j = \bar{f}(i)$};
\node[right] at ({cos(-30)},{sin(-30)}) {$i$};
\node[left] at (-{cos(-30)},{sin(-30)}) {$\bar{f}(j)$};
\end{scope}
\begin{scope}[xshift = 12 cm]
\draw (0,0) circle (1);
\draw[->, >=latex] ({cos(30)},{sin(30)}) to ({cos(-30)},{sin(-30)});
\draw[->, >=latex] ({cos(-30)},{sin(-30)}) to (-{cos(-30)},-{sin(-30)});
\node[right] at ({cos(30)},{sin(30)}){$j$};
\node[left] at (-{cos(30)},{sin(30)}) {$\bar{f}(i)$};
\node[right] at ({cos(-30)},{sin(-30)}) {$i=\bar{f}(j)$};
\end{scope}
\end{tikzpicture}
\caption{Crossings in a chord diagram.}
\label{crossings}
\end{figure}

Let $\Pi$ be a positroid cell in $\Gr_{\geq 0}(k,n)$ with bounded affine permutation $f$, and let $V$ be a matrix representing a point in $\Pi$.  Let $v_1,\ldots,v_n$ be the columns of $G$.  We extend this periodically to a sequence of vectors in $v_i \in \mathbb{C}^k$ by setting $v_i = v_j$ if $i \equiv j \pmod{n}$.  The following lemma, which is implicit in \cite{KLS13}, gives a characterization of $f$ in terms of the $v_i$.

\begin{lem}
With the notation above, $f(i)$ is the smallest $r \geq i$ with
\[v_i \in \Span(v_{i+1},v_{i+2},\ldots,v_r).\]
Similarly $f^{-1}(i)$ is the largest $r \leq i$ such that 
\[v_i \in \Span(v_r,v_{r+1},\ldots,v_{i-1}).\]\end{lem}

\subsection{Grassmann and reverse Grassmann necklaces}
 
We now introduce a final pair of indexing sets for positroid cells: Grassmann necklaces and dual Grassmann necklaces.  The combinatorics of Grassmann necklaces are vital to the proof of Theorem \ref{symbrd}, which shows that we can construct a symmetric plabic graph with symmetric weights for any point in the symmetric part of $\Gr_{\geq 0}(k,n)$.

\begin{defn}A \emph{Grassmann necklace} $\mathcal{I} = (I_1,\ldots,I_n)$ of type $(k,n)$ is a sequence of $k$-element subsets of $[n]$ such that the following hold, with indices taken modulo $n$:
\begin{enumerate}
\item If $i \in I_i$ then $I_{i+1} = (I_i \cup \{j\}) -\{i\}$ for some $j \in [n].$
\item If $i \not\in I_i$, then $I_{i+1} = I_i.$
\end{enumerate}
\end{defn}

Postnikov defined a combinatorial bijection between Grassmann necklaces of type $(k,n)$ and bounded affine permutations.
Let $\mathcal{I}$ be a Grassmann necklace of type $(k,n)$.  We construct the bounded affine permutation $f$ corresponding to $\mathcal{I}$ as follows, with indices taken modulo $n$:
\begin{enumerate}
\item If $I_{i+1} = (I_i \cup \{j\}) - \{i\}$ for some $j \neq i$, then $\bar{f}(i) = j$.
\item If $I_{i+1} = I_i$ and $i \not\in I_i$, then $i$ is a black fixed point of $f$.
\item If $I_{i+1} = I_i$ and $i \in I_i$, then $i$ is a white fixed point of $f$.
\end{enumerate}  

Next, we describe how to recover the Grassmann necklace $\mathcal{I}$ from the bounded affine permutation $f$.
We say $i \in [n]$ is an \emph{anti-exceedance} of $f$ if either $\bar{f}^{-1}(i) > i$ or $i$ is a white fixed point.  We say $i$ is an \emph{$a$-anti-exceedance} if we have $\bar{f}^{-1}(i) >_a i$ or $i$ is a white fixed point.  The Grassmann necklace $\mathcal{I}$ corresponding to $f$ is given by setting
\begin{equation}I_a = \{ i \in [n] \mid i \text{ is an $a$-anti-exceedance of $f$}\}.\end{equation}

Let $\mathcal{M}$ be a positroid of type $(k,n)$.  Then $\mathcal{M}$ is a collection of $k$-element subsets of $[n]$.  For each $1 \leq i \leq n$, let $I_i$ be the minimal element of $\mathcal{M}$ with respect to the shifted linear order $\leq_{i}$.  Then $\mathcal{M}$ is a Grassmann necklace of type $(k,n)$.  This procedure gives a bijection between Grassmann necklaces and positroids \cite{Pos06}.  For the inverse bijection, let $\mathcal{I}=(I_1,\ldots,I_n)$ be a Grassmann necklace of type $(k,n)$, and let $\mathcal{M}$ be the set of all $k$-element subsets $J \in {{[n]}\choose{k}}$ such that $I_i \leq_i J$ for all $i \in [n]$.  Then $\mathcal{M}$ is the positroid corresponding to $\mathcal{I}$ \cite{Oh11}.  

Let $\mathcal{M}$ and $\mathcal{I}$ be as above.  If $V$ is a matrix representing some point in the positroid cell $\Pi$ corresponding to $\mathcal{M}$, with columns $v_1,\ldots,v_n$, then the columns indexed by $I_i$ represent the  minimal basis for $\mathbb{C}^k$ among the columns of $V$ with respect the the cyclic order $\leq_i$.  We will use this fact frequently in the proof of Theorem \ref{symbrd}.

We also recall the dual notion, defined in \cite[Section 3.6]{Lam11}.

\begin{defn}
A \emph{dual Grassmann necklace} of type $(k,n)$ is a sequence $\mathcal{J} = (J_1,\ldots,J_n)$ of $k$-element subsets of $n$ such that the following hold, with indices taken modulo $n$:
\begin{enumerate}
\item If $i \in J_{i+1}$, then $J_{i} = (J_{i+1} \cup \{j\}) - \{i\}$ for some $j$
\item If $i \not\in J_{i+1}$, then $J_i = J_{i+1}$.
\end{enumerate}
\end{defn}

We have bijections between reverse Grassmann necklaces, decorated permutations, and positroids, which commute with the bijections given above for Grassmann necklaces.  We describe these briefly below. 

For $\mathcal{J}$ a dual Grassmann necklace, we define the corresponding bounded affine permutation $f$ by setting
\begin{enumerate}
\item If $J_{i} = (J_{i+1} \cup \{j\}) - \{i\}$ for some $j \neq i$, then $\bar{f}^{-1}(i) = j$.
\item If $J_i = J_{i+1}$ and $i \not\in J_{i+1}$, then $i$ is a black fixed point of $f$.
\item If $J_i = J_{i+1}$ and $i \in J_{i+1}$, then $i$ is a white fixed point of $f$
\end{enumerate}
where again indices are taken modulo $n$. 

If $\Pi$ is a positroid cell with positroid $\mathcal{M}$ and dual Grassmann necklace $\mathcal{J}$, then $J_i$ gives the maximal element of $\mathcal{M}$ with respect to the shifted order $\leq_i$.

\subsection{Bridge graphs}

We now describe a way to build up plabic graphs inductively. For more details on this construction, see \cite{Lam11}.  

Let $G$ be a reduced plabic graph with bounded affine permutation $f$ of type $(k,n)$, corresponding to a positroid cell $\Pi$.  If $f(i) > f(i+1)$ for some $i \in [n]$, then $fs_i$ is a bounded affine permutation of type $(k,n)$, with corresponding positroid cell $\Pi^*$.  Moreover, we have $\dim(\Pi^*) = \dim(\Pi) +1$.  We may add an edge, called a \emph{bridge}, between the two edges of $G$ incident  $i$ and $i+1$, to produce a graph $G^*$ corresponding to $\Pi^*$.  The bridge has a white vertex adjacent to $i$ and a black vertex adjacent to $i+1$.  If $G$ has a boundary leaf, or lollipop, at $i$ or $i+1$, that leaf becomes one endpoint of the bridge.  Note that after adding the bridge, we add degree-two vertices as needed to make the graph bipartite, as in Figure \ref{addbridge}.  

For $i \in [n]$, let $x_i(c)$ be the elementary matrix with $1$'s along the diagonal, a nonzero entry $c$ at position $(i,i+1)$, and $0$'s everywhere else.  Let $G$ be as above, and suppose $f(i) < f(i+1)$.  Assign positive real weights to the edges of $G$, and let $M$ be a matrix representing the corresponding point in $\Gr_{\geq 0}(k,n).$
Applying gauge transformations, we can assume the edges adjacent to boundary vertices $i$ and $i+1$ have weight $1$.  Adding a bridge at $(i,i+1)$ with weight $c$ corresponds to multiplying $M$ on the right by $x_i(c)$.  Let $X$ denote the point in $\Gr_{\geq 0}(k,n)$ corresponding to $G$ with its original weighting, and $X^*$ denote the point corresponding to the weighted graph obtained by adding the bridge.  The boundary measurements change as follows:

\begin{equation}\Delta_I(X^*) = \begin{cases}
\Delta_I(X) + c\Delta_{(I\cup \{i\}) - \{i+1\}}(X) & \text{ if $i+1 \in I$ but $i \not\in I$}\\
\Delta_I(X) & \text{otherwise}
\end{cases}.\end{equation}
Note that we are abusing notation slightly, since the $\Delta_I$ represent homogeneous coordinates rather than functions.

\begin{figure}[h]
\centering
\begin{tikzpicture}[scale = 0.7]
\draw (2,4) -- (2,0);
\draw (-1,3.5) -- (0,3);
\draw (-1,2.5) -- (0,3);
\draw (-1,1.5) -- (0,1);
\draw (-1,0.5) -- (0,1);
\draw (0,3) -- (2,3); \wdot{0}{3}; \bdot{2}{3};
\draw (0,1) -- (2,1); \bdot{0}{1}; \wdot{2}{1};
\draw [->] (3,2) [out = 30, in = 150] to (5,2);
\begin{scope}[xshift = 8 cm]
\draw (2,4) -- (2,0);
\draw (-2,3.5) -- (-1,3);
\draw (-2,2.5) -- (-1,3);
\draw (-2,1.5) -- (-1,1);
\draw (-2,0.5) -- (-1,1);
\draw (-1,3) -- (2,3); \wdot{-1}{3}; \bdot{2}{3};
\draw (-1,1) -- (2,1); \bdot{-1}{1}; \wdot{2}{1};
\edge{0.5}{1}{2};
\draw [->] (3,2) [out = 30, in = 150] to (5,2);
\end{scope}
\begin{scope}[xshift = 16 cm]
\draw (2,4) -- (2,0);
\draw (-2,3.5) -- (-1,3);
\draw (-2,2.5) -- (-1,3);
\draw (-2,1.5) -- (-1,1);
\draw (-2,0.5) -- (-1,1);
\draw (-1,3) -- (2,3); \wdot{-1}{3}; \bdot{2}{3};
\draw (-1,1) -- (2,1); \bdot{-1}{1}; \wdot{2}{1};
\edge{0.5}{1}{2}; \bdot{-0.25}{3}; \wdot{-0.25}{1};
\end{scope}
\end{tikzpicture}
\caption{Adding a bridge to a plabic graph.}
\label{addbridge}
\end{figure}
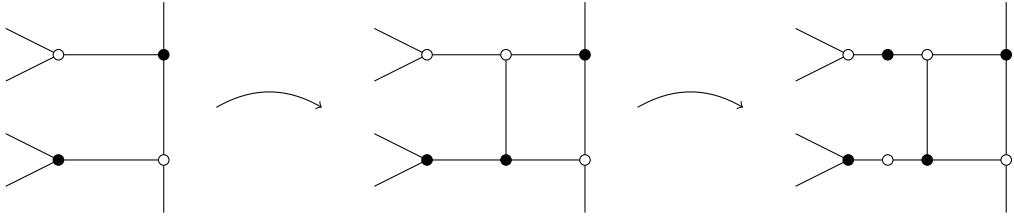

The following proposition, from \cite{Lam13b}, gives a way to ``undo" the operation of adding a bridge, at least on the level of matrix representatives.

\begin{prop}\label{unbridge}
Let $X \in \Gr_{\geq 0}(k,n)$, and suppose $X$ is contained in the positroid cell $\Pi$ with bounded affine permutation $f$. 
Suppose $i < f(i) < f(i+1) < i+n+1$.  Then
\begin{equation}c  = \frac{\Delta_{I_{i+1}}(X)}{\Delta_{(I_{i+1}\cup \{i\})-\{i+1\}}(X)}\end{equation}
is positive and well defined, and $X^* = X \cdot x_i(-c)$ is in $\Pi^*$, where $\Pi^* \subset \Gr_{\geq 0}(k,n)$ is the positroid with bounded affine permutation $fs_i$ and $\dim(\Pi^*) = \dim(\Pi)-1$.
\end{prop}

Note  that we have 
\begin{equation}i < f(i) < f(i+1) < i+n+1\end{equation} if and only if the chords in the chord diagram for $f$ which start at $i$ and $i+1$ respectively form a crossing as in Figure \ref{crossings}.  If this occurs, we say that $f$ \emph{has a bridge} at $(i,i+1)$.

\section{Characterizing symmetric plabic graphs}

Let $G$ be a plabic graph with $2n$ boundary vertices, corresponding to a positroid cell $\Pi$.  Label the boundary vertices of $G$ with the numbers $1,2,\ldots,2n$ in clockwise order, and fix a distinguished diameter $d$ of $G$ with one end between vertices $2n$ and $1$, and the other between vertices $n$ and $n+1$.  For each $i \in 2n$, let $i' = 2n+1-i$.

For $I=\{i_1,i_2,\ldots,i_k\}$, we define $R(I)=[2n] \backslash\{i' \mid i \in I\}.$
Let $G$ be a symmetric plabic graph, and let $I \subseteq [2n]$.  Reflection through $d$ gives a bijection between the set of almost perfect matchings $P$ of $G$ with $\partial(P)=I$ and the set of almost perfect matchings $P'$ of $G$ with $\partial(P')=R(I)$. One immediate consequence is that if $I = \partial(P)$ for some matching $P$ of $G$ then $|I|=n$, so $\Pi$ lies in $Gr_{\geq 0}(n,2n)$. 
 
\begin{figure}
\centering
\begin{tikzpicture}
\draw [gray!30] (0,3.5) -- (0,-3.5);
\draw (0,0) circle (3);
\draw ({2*cos(60)},{2*sin(60)}) -- ({-2*cos(60)},{2*sin(60)})  -- ({-2*cos(60)},{-2*sin(60)})  -- ({2*cos(60)},{-2*sin(60)}) --  ({2*cos(60)},{2*sin(60)});
\draw ({2*cos(60)},{2*sin(60)}) -- ({3*cos(60)},{3*sin(60)});
\draw ({-2*cos(60)},{2*sin(60)}) -- ({-3*cos(60)},{3*sin(60)});
\draw ({-2*cos(60)},{-2*sin(60)}) -- ({-3*cos(60)},{-3*sin(60)});
\draw ({2*cos(60)},{-2*sin(60)}) -- ({3*cos(60)},{-3*sin(60)});
\draw ({3*cos(30)},{3*sin(30)}) -- (2,0) -- ({3*cos(30)},{-3*sin(30)});
\draw ({-3*cos(30)},{3*sin(30)}) -- (-2,0) -- ({-3*cos(30)},{-3*sin(30)});
\bdot{{2*cos(60)}}{{2*sin(60)}};\wdot{{-2*cos(60)}}{{2*sin(60)}};\bdot{{-2*cos(60)}}{{-2*sin(60)}};\wdot{{2*cos(60)}}{{-2*sin(60)}};
\wdot{{3*cos(60)}}{{3*sin(60)}};\bdot{{3*cos(30)}}{{3*sin(30)}};\wdot{2}{0};\bdot{{3*cos(30)}}{{-3*sin(30)}};\bdot{{3*cos(60)}}{{-3*sin(60)}};
\bdot{{-3*cos(60)}}{{3*sin(60)}};\wdot{{-3*cos(30)}}{{3*sin(30)}};\bdot{-2}{0};\wdot{{-3*cos(30)}}{{-3*sin(30)}};\wdot{{-3*cos(60)}}{{-3*sin(60)}};
\node [above right] at ({3*cos(60)},{3*sin(60)}) {1};
\node [above right] at ({3*cos(30)},{3*sin(30)}) {2};
\node [below right] at ({3*cos(-30)},{3*sin(-30)}) {3};
\node [below right] at ({3*cos(-60)},{3*sin(-60)}) {4};
\node [below left] at ({3*cos(240)},{3*sin(240)}) {5};
\node [below left] at ({3*cos(210)},{3*sin(210)}) { 6};
\node [above left] at ({3*cos(150)},{3*sin(150)}) {7};
\node [above left] at ({3*cos(120)},{3*sin(120)}) {8};
\node [right] at (0,2) {4};
\node [right] at (0,-2) {3};
\node [right] at (1,0) {7};
\node [left] at (-1,0) {7};
\end{tikzpicture}
\caption{A symmetric weighting of a symmetric plabic graph. The distinguished diameter $d$ is shown in gray.  Unlabeled edges have weight $1$.}
\label{graph}
\end{figure}
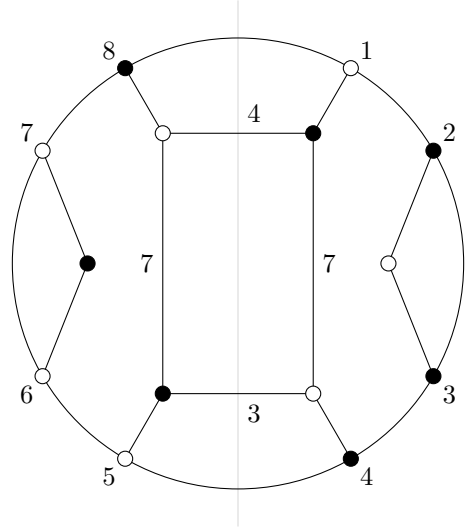

\begin{thm}
\label{graphsLagrangian}
Let $\Pi$ be a positroid cell in $Gr_{\geq 0}(n,2n)$ with positroid $\cM$ and bounded affine permutation $f$. Let $\mathcal{I} = (I_1,\ldots,I_{2n})$ be the Grassmann necklace of $\Pi$, and let $\mathcal{J} = (J_{1},\ldots,J_{2n})$ be the dual Grassmann necklace.
\label{positroidiwthsym2}
The following are equivalent:
\begin{enumerate}[(i)]
\item\label{graph1} $\Pi$ can be represented by a symmetric plabic graph.
\item\label{matroid} $I\in\cM$ if and only if $R(I)\in\cM$.
\item\label{perm} For $a \in [2n]$, if $f(a) = b$, then $f(2n+1-a)=4n+1-b$.
\item\label{necklace} $R(I_i)=I_{i'+1}$ for all $i\in [2n]$ with indices taken modulo $2n$.
\item\label{dualnecklace} $R(J_i)=J_{i'+1}$ for all $i\in [2n]$ with indices taken modulo $2n$.
\end{enumerate}
\end{thm}

\begin{proof}

The above discussion shows that \eqref{graph1} implies \eqref{matroid}.

Next, we show that \eqref{matroid} and \eqref{perm} are equivalent.  Consider a reduced plabic graph $H$ corresponding to $\Pi$, not necessarily symmetric, but with distinguished diameter $d$ as above.  Reflect $H$ about $d$, and reverse the colors of all vertices.  Let $H'$ be the resulting plabic graph.  We note that $H'$ is reduced.  Indeed, suppose $H$ has a trip from boundary vertex $a$ to boundary vertex $t$.  Then this trip corresponds, under reflection though $d$ and reversing colors of vertices, to a trip from $a'$ to $t'$ in $H'$.  Hence, $H'$ satisfies the criterion for being reduced, which is phrased entirely in terms of forbidden intersections between trips.

Let $\bar{f}$ be the trip permutation of $H$, and $\bar{f}'$ the decorated permutation of $H'$.  Then $\bar{f}(a) = t$ implies $\bar{f}'(a') = t'$.  Note also that $a < t$ if and only if $a' > t'$.  Moreover, $a$ is a black fixed point of $f$ if and only if $a'$ is a white fixed point of $f'$, and vice versa.  It follows that $f(a) = b$ implies 
\begin{equation}f'(2n+1-a) = 4n+1-b.\end{equation}

The reduced plabic graph $H'$  corresponds to some positroid cell $\Pi'$, with positroid $\cM'$.  It is clear that $\cM' = \{I \mid R(I) \in \cM\}$.  Hence $\cM' = \cM$ if and only if $\cM$ satisfies condition \eqref{matroid} above, while $\bar{f} = f'$ if and only if $f$ satisfies condition \eqref{perm}.  But the statement $\cM' = \cM$ and $\bar{f} = f'$ are both, in turn, equivalent to $\Pi = \Pi'$.  Hence \eqref{perm} and \eqref{matroid} are equivalent as desired.  

Next, we show \eqref{perm} implies \eqref{graph1}.  Suppose $f(a) = b$ implies $f(2n+1-a) = 4n+1-b$.  
We claim that we can construct a symmetric plabic graph with bounded affine permutation $f$, using the bridge graph construction given above.
We proceed by induction on the dimension of the positroid $\Pi$ corresponding to $G$.  Note first that if $a$ is a black fixed point of $f$, then \eqref{perm} implies that $a'$ is a white fixed point, and vice versa.  We may therefore assume that $f$ has no fixed points.  Hence there is some pair $(i,i+1)$ with $1 \leq i \leq 2n$ such that the bounded affine permutation $f$ has a bridge at $(i,i+1)$.  If $i = n$, let $g = fs_i$.  Otherwise, let $g = fs_{i'-1}s_i$.  In each case, $g$ is a bounded affine permutation which satisfies \eqref{perm}.  By induction, we can build a symmetric plabic graph corresponding to $g$.  Adding a bridge at $(n,n+1)$ or a pair of bridges at $(i,i+1)$ and $(i'-1,i')$, respectively, we obtain a symmetric plabic graph for $f$.

Next, we show that \eqref{perm} implies \eqref{necklace} and \eqref{dualnecklace}.  Indeed, if $\eqref{perm}$ holds, then the chord diagram for $f$ has an arrow from $a$ to $t$ if and only if it has an arrow from $a'$ to $t'$.  Hence for each $i \in [2n]$, we have $\bar{f}^{-1}(a) >_i a$ if and only if $\bar{f}^{-1}(a') <_{(i'+1)}a',$ and so $I_{i} = R(I_{i'+1})$.  The argument for the dual Grassmann necklace is analogous.  

Finally, we show that \eqref{necklace} and \eqref{dualnecklace} each imply \eqref{perm}.  
Suppose \eqref{necklace} holds.  
Note first that for $a \in [2n]$, we have $a \in I_i$ if and only if $a' \notin I_{i'+1}$.  It follows in particular that if $a \in I_i$ for all $i \in [2n]$ (that is, $a$ corresponds to a ``white'' fixed point of $f$) then $a' \not\in I_i$ for all $i$, and $a'$ corresponds to a black fixed point.  Hence $f(a') = 4n+1-f(a)$ as desired.  The analogous argument holds if $a$ is a black fixed point.  

Suppose $a$ is not a fixed point of $f$.  
Then we have
\begin{align} I_a &= (I_{a+1} \cup \{a\}) - \{\bar{f}(a)\}\\
I_{a'+1} & = (I_{a'} \cup \{\bar{f}(a')\}) - \{a'\}
\end{align}
Applying $R$ to the first line above, we have
\begin{equation}R(I_a) = (R(I_{a+1}) \cup \{(\bar{f}(a))'\}) - \{a'\}\end{equation}
By \eqref{necklace}, this implies
\begin{equation}I_{a'+1} = (I_{a'} \cup \{(\bar{f}(a))'\}) -\{a'\}\end{equation}
Comparing this with our previous expression for $I_{a'+1}$ above, we have
\begin{equation}\bar{f}(a')=(\bar{f}(a))'\end{equation}
It follows that $f$ is a bounded affine permutation of type $(n,2n)$ which satisfies $\eqref{perm}$.  Again, note that we could make an analogous argument using the dual Grassmann necklace.  This completes the proof.
\end{proof}

\section{Network Realizations of Symmetric Points}

Let $G$ be a symmetric plabic graph.  Suppose we assign weights to the edges of $G$ such that each edge $(u,v)$ has the same weight as its reflection $(r(u),r(v))$ about the distinguished diameter $d$.  For $I \in {{[n]}\choose{k}}$, reflection across $d$ then gives a weight-preserving bijection between almost perfect matchings $P$ with $\partial(P) = I$ and almost perfect matchings $P'$ with $\partial(P')=R(I)$.

\begin{defn}A point $X \in \Gr_{\geq 0}(n,2n)$ is a \emph{symmetric point} if  $\Delta_{I}(P)=\Delta_{R(I)}(P)$ for all $I$.  The subvariety of symmetric points in $\Gr_{\geq 0}(n,2n)$  is the \emph{symmetric part} of $\Gr_{\geq 0}(n,2n)$.  Similarly, for $\Pi$ a positroid cell, the \emph{symmetric part} of $\Pi$ is the intersection of $\Pi$ with the symmetric part of $\Gr_{\geq 0}(n,2n)$.  
\end{defn}

Clearly, the image of every symmetric weighting of a symmetric plabic graph is a symmetric point.  We show the converse.

\begin{thm}\label{allpt}Let $X \in \Gr_{\geq 0}(n,2n)$ be a symmetric point.  Then $X$ may be represented by a symmetric plabic graph with symmetric weights.
\end{thm}

\begin{proof}Since $X$ is a symmetric point, the positroid of $X$ satisfies condition (\ref{matroid}) from Theorem \ref{graphsLagrangian}.  Hence $X$ may be represented by some weighting $\omega$ of a symmetric plabic graph $G$.  Choose a collection $F$ of edges of $G$ which has all of the following properties:
\begin{enumerate}
\item $F$ is symmetric about the diameter $d$.  That is, an edge $(u,v)$ of $G$ is in $F$ if and only if $(r(u),r(v))$ is also in $F$. 
\item \label{cover} Every vertex of $G$ is covered by $F$ at least once.
\item \label{trees} $F$ consists of a disjoint collection of trees, each of which has exactly one vertex on the boundary of $G$.
\end{enumerate}
It is not hard to show that such a collection of edges exists, since $F$ is symmetric.
Since $F$ is a disjoint union of trees, each with exactly one vertex on the boundary, we may gauge fix all edges in $F$ to $1$.  Let $\nu$ be the resulting weighting of $G$, and let $\nu'$ be the weighting of $G$ obtained by swapping the weights of $(u,v)$ and $(r(u),r(v))$ for each edge $(u,v)$ of $G$.  

By symmetry, $\nu'$ is another weighting of $G$ with $\partial_G(\nu') = X$.  It follows that $\nu$ and $\nu'$ are the same up to gauge transformations.
Moreover, $\nu'$ assigns the weight $1$ to each edge of $F$.  Since $F$ satisfies conditions \ref{cover} and \ref{trees} above, a sequence of gauge transformations which fixes all edges in $F$ must be trivial.  Hence $\nu = \nu'$ and $\nu$ is a symmetric weighting.  This completes the proof.
\end{proof}

The following is immediate.

\begin{cor}\label{both}Let $\Pi$ be a positroid cell.  Then the symmetric part of $\Pi$ is nonempty if and only if $\Pi$ can be realized as the image of a symmetric plabic graph $G$, where the symmetric part of $\Pi$ is precisely the image of all symmetric weightings of $G$ under the boundary measurement map.
\end{cor}

A \emph{lollipop graph} is a plabic graph which has a lollipop at each boundary vertex.  A \emph{bridge graph} is a plabic graph obtained by starting with a lollipop graph, and repeatedly adding bridges adjacent to the boundary.  See Figure \ref{symbrd} for an example. Given $X \in \Gr_{\geq 0}(k,n)$, we can apply Proposition \ref{unbridge} repeatedly to construct a weighted plabic graph whose image under the boundary measurement map is the point $X$ \cite{Lam13b}.  We now give a symmetric version of this result.

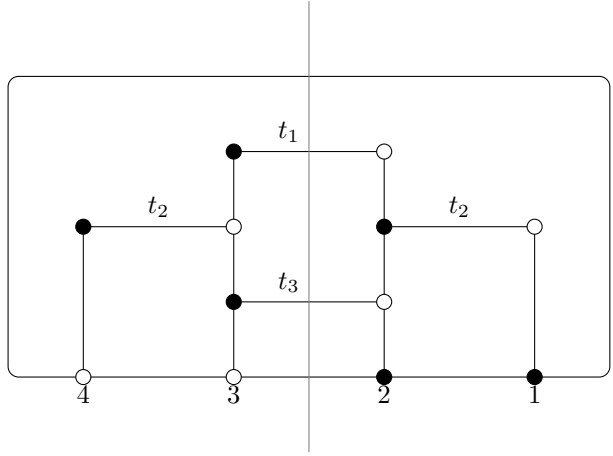
\begin{figure}[ht] 
\centering
\begin{tikzpicture}
\draw [rounded corners] (0,0) rectangle (8,4);
\draw (3,0) -- (3,3) -- (5,3) -- (5,0);
\draw (1,0) -- (1,2) -- (3,2);
\draw (5,2) -- (7,2) -- (7,0);
\draw (3,1) -- (5,1);
\bdot{1}{2};\bdot{3}{3};\bdot{5}{2};\bdot{3}{1};
\wdot{3}{2};\wdot{5}{3};\wdot{5}{1};\wdot{7}{2};
\wdot{1}{0};\wdot{3}{0};\bdot{5}{0};\bdot{7}{0};
\node [below] at (1,0) {$4$};
\node [below] at (3,0) {$3$};
\node [below] at (5,0) {$2$};
\node [below] at (7,0) {$1$};
\draw [gray] (4,5) -- (4,-1);
\node [above left] at (4,3) {$t_1$};
\node [above] at (2,2) {$t_2$};
\node [above] at (6,2) {$t_2$};
\node [above left] at (4,1) {$t_3$};
\end{tikzpicture}
\caption{A symmetric bridge graph with symmetric weights.}
\label{symbrd}
\end{figure}

\begin{thm}
\label{symbrd}
Let $X \in \Gr_{\geq 0}(k,n)$ be a symmetric point.  Then we can iteratively construct a weighted bridge graph corresponding to $X$ which is symmetric and has symmetric weights.
\end{thm}

\begin{proof}
Let $\Pi$ be the positroid cell containing $X$, and let $f$ be the decorated permutation of $\Pi$.  We induct on the dimension of $\Pi$.  If the dimension is $0$, then $X$ may be represented by a lollipop graph $G$ \cite{Pos06}.  In this case, $X$ has a single non-zero Pl\"{u}cker coordinate $I$, corresponding to the positions of the white lollipops, which necessarily satisfies $I = R(I)$.  Thus $G$ has a white lollipop at $i$ if and only if $G$ has a black lollipop at $i'$, and $G$ is symmetric.

Suppose $\dim(\Pi) > 0$.  We note first that $\Pi$ satisfies condition (\ref{matroid}) from Theorem \ref{graphsLagrangian}.  Hence $i$ is a white fixed point of $f$ if and only if $i'$ is a black fixed point of $f$.  We may therefore reduce to the case where $f$ has no fixed points.  It follows that $f$ has a bridge at $(i,i+1)$ for some $1 \leq i \leq 2n-1$.  

First, suppose $f$ has a bridge at $(n,n+1)$.  Let $(I_{1},\ldots,I_{2n})$ denote the Grassmann necklace of $X$.  By Proposition \ref{unbridge}, the quantity
\begin{equation}c = \Delta_{I_{n+1}}(X)/\Delta_{(I_{n+1} \cup \{n\}) - \{n+1\}}(X)\end{equation}
is positive and well-defined, and the point $X^* = X \cdot x_{n}(-c)$ lies in $\Gr_{\geq 0}(n,2n).$

The action of $x_n(-c)$ fixes the Pl\"{u}cker coordinates of $X,$ except for those $\Delta_I$ with $n+1 \in I$ and $n \not\in I$.  Note that $I \in {{[2n]}\choose{n}}$ has this property if and only if $R(I)$ does.  For each $I$ with $n+1\in I$ and $n \not\in I$, we have
\begin{align}
\Delta_{R(I)}(X^*) &= \Delta_{R(I)}(X) - c \Delta_{(R(I) \cup \{n\}) - \{n+1\}}(X)\\
&= \Delta_I(X) - c \Delta_{R((I \cup \{n\} )- \{n+1\})}(X)\\
&= \Delta_I(X) - c \Delta_{(I \cup \{n\})-\{n+1\}}(X)\\
& = \Delta_I(X^*)
\end{align}
So $X^*$ is a symmetric point.
By induction, we may build a symmetric bridge graph representing the point $X^*$, which has symmetric weights.  Adding a bridge $(n,n+1)$ with edge weight $c$ gives the desired weighted graph, and we are done in this case.

Next, suppose $f$ does not have a bridge at $(n,n+1)$.  Then $f$ has a bridge $(i,i+1)$ for some $1 \leq i \leq 2n-1$ with $i \neq n$.  Hence by symmetry, $f$ has a pair of commuting bridges at $(i,i+1)$ and $(i'-1,i'),$ respectively.  Without loss of generality, assume $1 \leq i \leq n-1$.
By Theorem 2.1, we have $I_{i'} = R(I_{i+1})$, which implies
\begin{equation} (I_{i'} \cup \{i'-1\}) - \{i'\} = R((I_{i+1} \cup \{i\}) - \{i+1\}).\end{equation}
Since $X$ satisfies $\Delta_{R(I)}(X) = \Delta_I(X)$ for all $I$, we have
\begin{equation}c \coloneqq \Delta_{I_{i+1}}(X)/\Delta_{(I_{i+1} \cup \{i\}) - \{i+1\}}(X) = \Delta_{I_{i'}}(X)/\Delta_{(I_{i'} \cup \{i'-1\}) - \{i'\}}(X)\end{equation}
where both quantities are positive and well-defined, as above.  

Once again, let $X^* = X \cdot x_i(-c)$.  Let $(I^*_{1},\ldots,I^*_{2n})$ denote the Grassmann necklace of $X^*$, and let 
\begin{equation}c^* = \Delta_{I^*_{i'}}(X^*)/\Delta_{(I^*_{i'} \cup \{i'-1\})- \{i'\}}(X^*).\end{equation}  We claim that $c^* = c$. 
First, note that $I^*_j = I_j$ unless $j = i+1$, so $I^*_{i'} = I_{i'}$.  If either $i \in I_{i'}$ or $i+1 \not\in I_{i'}$, then we have
\begin{equation}\Delta_{I_{i'}}(X) = \Delta_{I_{i'}}(X^*)\end{equation}
\begin{equation}\Delta_{(I_{i'} \cup \{i'-1\}) - \{i'\}}(X)=\Delta_{(I_{i'} \cup \{i'-1\}) - \{i'\}}(X^*),\end{equation}
and so $c^* = c$, as desired.  

Suppose $i \not \in I_{i'}$ and $i+1 \in I_{i'}$.  Let $M$ be a matrix representative for $X$ with columns $v_{1},\ldots,v_{2n}$.  Suppose $f(i'-1) = i'$.  Then $v_{i'}$ is a scalar multiple of $v_{i'-1}$, and in particular we have $v_{i'} = cv_{i'-1}$.  Let $v_1^*,\ldots,v_{2n}^*$ be the columns of $M^* = M \cdot x_i(-c)$.  Then $v_{i'}^*= v_{i'}$ and $v_{i'-1}^*=v_{i'-1}$.  It follows that we have $c^* = c$ as desired.

Next, suppose $f(i'-1) \neq i'$, and consider the submatrix $S$ of $M$ with columns indexed by $I_{i'}$.  Since $i \not\in I_{i'}$, it follows that $v_i$ is in the span of the columns of $S$ indexed by elements of $[i',i-1]^{cyc}$.  Hence adding a multiple of $v_i$ to the column $v_{i+1}$ does not change the determinant of $S$, and $\Delta_{I_{i'}}(X^*)= \Delta_{I_{i'}}(X)$.  Thus the ratios $c^*$ and $c$ have the same numerator.  

We must now show that $c^*$ and $c$ have the same denominator.  Since $f(i'-1) \neq i',$ we have $i' \in I_{i'-1} \cap I_{i'}$.  It follows that
\begin{equation} (I_{i'} \cup \{i'-1\}) - \{i'\} = (I_{i'-1} \cup \{\bar{f}(i'-1)\}) - \{i'\}.\end{equation}
Let $S^*$ be the submatrix of $M$ with columns indexed by this set.
There are two cases to consider, depending on whether $\bar{f}(i')$ lies in the cyclic interval $[i',i]^{cyc}$.  Throughout this proof, we write $v_a,\ldots,v_b$ to denote the columns of $S$ indexed by elements of $[a,b]^{cyc}$.

For the first case, suppose $\bar{f}(i') \in [i',i]^{cyc}.$ Since $G$ has a bridge at $(i'-1,i')$ and $\bar{f}(i'-1) \neq i'$, this means $\bar{f}(i'-1)$ must be in the cyclic interval $[i'+1,i-1]^{cyc}$.  Let $y = \bar{f}(i'-1)$.
Then we may express $v_y$ as a linear combination of the columns indexed by $I_{i'-1}$ which form a basis of $\langle v_{i'-1},\ldots,v_{y-1}\rangle$.
Note that we must have a nonzero coefficient of $v_{i'}$ in this linear combination, since the columns of $S^*$ are linearly independent.

Hence $v_{i'}$ lies in the the span of the columns of $S^*$ indexed by elements of $[i'-1,y]^{cyc}$.  In particular, $v_{i'}$ lies in the span of the columns of $S^*$ indexed by elements of $[i'-1,i-1]^{cyc},$ which therefore span $\langle v_{i'-1},v_{i'},\ldots,v_{i-1} \rangle$.  Thus $v_i$ lies in the span of the columns of $S^*$ indexed by elements of $[i'-1,y]^{cyc}$ and we are done with this case.

For the second case, suppose $\bar{f}(i')$ does not lie in the cyclic interval $[i',i]^{cyc}$.  Then $v_{i'}$ is linearly independent of the columns 
$v_{i'+1},\ldots,v_i,$ so $(I_{i'} \cup \{i'-1\})-\{i'\}$ contains a basis for $\langle v_{i'+1},\ldots,v_{i-1} \rangle$.  Since $v_i \not\in I_{i'}$, the corresponding columns of $S^*$ contain $v_i$ in their span, and this case is complete.

Hence $c^* = c$, and so $X^{**} = X \cdot x_i(-c)x_{i'-1}(-c)$ lies in a positroid cell of dimension two less than $\Pi$.  It is straightforward to check that $X^{**}$ is a symmetric point. By induction, we may build a symmetric plabic graph for $X^{**}$ with symmetric edge weights.  Adding two bridges of weight $c$ at $(i'-1,i')$ and $(i,i+1)$ respectively gives a symmetric plabic network for $X$, and the proof is complete. 

\end{proof}

Note that the sequence of bridges added above depends only on the bounded affine permutation of $\Pi$, and the fact that $X$ is symmetric.  Hence, the proof of Theorem \ref{symbrd} yields a method for constructing a symmetric bridge graph $G$ corresponding to $\Pi$, and explicitly realizing each point in the symmetric part of $\Pi$ with a symmetric weighting of $G$.  We have thus demonstrated a symmetric analog of the bridge-graph construction for $\Gr_{\geq 0}(k,n)$ found in \cite{Lam13b}.

\section*{Acknowledgements} 

We are grateful to Francesca Gandini for helpful suggestions.

\bibliographystyle{plain}
\bibliography{combo}

\end{document}